\documentclass[12 pt]{amsart}
\usepackage{amsfonts,amsmath,amsthm,amssymb, amscd}
\usepackage[top=2cm, bottom=5cm, left=3cm, right=1cm]{geometry}

\usepackage{tikz-cd}
\usepackage{xcolor}
\usepackage{placeins}
\usepackage{hyperref}
\usepackage[all]{xy}

\newtheorem{theorem}{Theorem}[section]
\theoremstyle{definition}
\newtheorem{conjecture}[theorem]{Conjecture}
\newtheorem{corollary}[theorem]{Corollary}
\newtheorem{lemma}[theorem]{Lemma}
\newtheorem{proposition}[theorem]{Proposition}
\newtheorem{question}[theorem]{Question}

\newtheorem{definition}[theorem]{Definition}
\newtheorem{example}[theorem]{Example}

\usepackage{lineno}
\usepackage{color}

\title[Orders of products of horizontal class transpositions]
{
Orders of products of horizontal class transpositions}
\author[V.~G.~Bardakov, A.~L.~Iskra]{V.~G.~Bardakov, A.~L.~Iskra}

\address{Sobolev Institute of Mathematics, 4 Acad. Koptyug avenue, 630090, Novosibirsk, Russia.}
\address{Saint Petersburg University, 7/9 Universitetskaya nab., St. Petersburg, 199034 Russia.}
\address{Novosibirsk State Agrarian University, Dobrolyubova street, 160, Novosibirsk, 630039, Russia.}
\address{Regional Scientific and Educational Mathematical Center of Tomsk State University, 36 Lenin Ave., Tomsk, Russia.}
\email{bardakov@math.nsc.ru}

\address{Novosibirsk State  University, 2 Pirogova Street, 630090, Novosibirsk, Russia.}
\address{Saint Petersburg University, 7/9 Universitetskaya nab., St. Petersburg, 199034 Russia.}
\email{a.iskra@g.nsu.ru}
\date{\today}

\begin{document}
\maketitle
\begin{abstract}
The class transposition group
$CT(\mathbb{Z})$ was introduced by S. Kohl in 2010. It is a countable subgroup  of
the permutation group $Sym(\mathbb{Z})$ of the set of integers $\mathbb{Z}$. 
 We study products of two class transpositions $CT(\mathbb{Z})$ and give a partial answer to the question 18.48 posed by S. Kohl in the Kourovka notebook. We prove that in the group $CT_{\infty}$, which is a subgroup of $CT(\mathbb{Z})$ and generated by horizontal class transpositions, the order of the product of a pair of horizontal class transpositions
belongs to the set $\{1,2,3,4,6,12\}$, and any number from this set is the order of the product of a pair of horizontal class transpositions.


\end{abstract}
\maketitle
\tableofcontents

\section{Introduction}
The problem of describing the permutation groups of a countable set  is one of the central problems of group theory. Just as any finite group can be embedded in some subgroup of the permutation  group of a finite set, any countable group can be embedded in the permutation  group of a countable set, in particular, in the permutation group $Sym(\mathbb{Z})$ of the set of integers $\mathbb{Z}$, which is already uncountable. Therefore, its countable subgroups are more often studied. One such group is the group $CT(\mathbb{Z})$ introduced by S. Kohl's \cite{K}. This group has a number of remarkable properties. In particular, it is simple and contains all finite subgroups. 
The group $CT(\mathbb{Z})$ is generated by transpositions of residue  classes of integers by different modules (briefly class transpositions). S. Kohl wrote  in the Kourovka notebook \cite[question 18.48]{Kour} the problem of describing the orders of products of two class transpositions (generators of $CT(\mathbb{Z})$). He also established the connection between this problem and the famous Collatz~\cite{K1} conjecture.  

S. Kohl's question is closely related to the question about the description of 2--generated subgroups of the group $CT(\mathbb{Z})$.
It is clear that if we consider 2--generated subgroups of  $Sym(\mathbb{Z})$, we cannot expect a complete description of them, in view of the well-known theorem of Higman, B. Neumann and H. Neumann \cite[Chapter IV, \S~3]{LS}, which states that any countable group can be embedded in a group generated by two elements of infinite order. 
The situation changes dramatically if we require that the permutations $t$, $q \in Sym(\mathbb{Z})$ are involutions. In this case, the subgroup generated by $t$ and $q$ is either infinite --- the free product of two cyclic groups of order 2, or a finite dihedral group. 

In this paper we study subgroups of the group $CT(\mathbb{Z})$ generated by two class transpositions. We are interested in the following question: is it true that there are only a finite number of such subgroups up to isomorphism? A positive answer to this question gives the a positive answer to S. Kohl's question. We give a geometric interpretation of any class transposition of $CT(\mathbb{Z})$, show that every such class transposition is either horizontal or oblique, 
for  every pair of class transpositions $\tau_1, \tau_2$ construct  a graph and prove in Theorem 1 that every connected component of this  graph corresponds to  one or two orbits of the product $\tau_1 \cdot \tau_2$.

We introduce a subgroup $CT_{\infty}$ of the group $CT(\mathbb{Z})$, which is  generated by horizontal class transpositions. We prove in  Proposition 3 that this group is isomorphic to  the group $CT_{int}(\mathbb{Z})$ (see \cite{K}), which is  generated by integer class transpositions. We show in Proposition 2  that the study of products of horizontal class transpositions reduces to the study of the product of permutations of some finite set. 
Theorem 2  describes all connected components of the graph, which corresponding to the product of a pair of horizontal class transpositions. From this theorem follows the main 
 result of the paper, which says that the order of the product of a pair of horizontal class transpositions 
belongs to the set $\{1,2,3,4,6,12\}$ and, conversely, for any number in this set there exists a pair of horizontal class transpositions whose product has the order which is  equal to this number. This answers S. Kohl's question for a pair of horizontal class transpositions.

The paper concludes with questions for further research.

In this paper, the set of natural numbers means the set of positive integers.
\bigskip

%%%%%%%%%%%%%%%%%%%%%%%%%%%%%%%%%%

\section{Preliminaries }
For a pair of natural numbers $r$ and $m$ such that 
 $0 \leq r < m$, denote the class of integers comparable to $r$ modulo $m$ by the symbol $r(m)$, i.e.
$$
r(m) = r + m \mathbb{Z} = \{ r + km~|~k \in \mathbb{Z} \}.
$$
For $r_1(m_1) \cap r_2(m_2) = \emptyset$ define the class transposition $\tau_{r_1(m_1),r_2(m_2)}$ as an involution which interchanges  $r_1+km_1$ and $r_2+km_2$ for each integer $k$ and  fixes everything else. The symbol $CT(\mathbb{Z})$ denotes the group generated by class transpositions (see \cite{K}).

We write the class transposition $\tau_{r_1(m_1), r_2(m_2)}$ as follows
$$
\tau_{r_1(m_1),r_2(m_2)} = \prod_{k \in \mathbb{Z}} (r_1 + m_1 k, r_2 + m_2 k).
$$

In this paper we will use the following geometric interpretation. Let us compare the class transposition $\tau = \tau_{r_1(m_1),r_2(m_2)}$, $r_1 \leq r_2$ the segment in the plane connecting the point $A = (r_1, m_1)$ with the point $B = (r_2, m_2)$. We will call the points $A$ and $B$ {\it vertices} of the class transposition $\tau$. If we want to emphasize that $A$ and $B$ are vertices of the class transposition $\tau$, we write $A = A(\tau)$ and $B = B(\tau)$.
With this interpretation we introduce

\begin{definition}
A permutation $\tau_{r_1(m_1), r_2(m_2)}$ is called {\it vertical} if $r_1 = r_2$.  A permutation $\tau_{r_1(m_1), r_2(m_2)}$ is called {\it horizontal} if $m_1 = m_2$. If  a permutation is neither vertical nor horizontal, it is called {\it oblique}.
\end{definition}

\begin{lemma}\label{LD}
A permutation $\tau_{r_1(m_1), r_2(m_2)}$ is a class transposition if and only if gcd($m_1, m_2$) does not divide $|r_1 - r_2|$.
In particular, a vertical permutation is not a class transposition.
\end{lemma}

\begin{proof}
By definition, a permutation $\tau_{r_1(m_1), r_2(m_2)}$ is a class transposition if $r_1(m_1) \cap r_2(m_2) = \emptyset$. This means that the Diophantine equation $r_1 + m_1 k = r_2 + m_2 l$ is not solvable for any integers $k$ and $l$. Since this equation is equivalent to the equation $r_1 - r_2 = m_2 l - m_1 k$, we get the desired statement.
\end{proof}

The following question can be found in the Kourovka  notebook \cite[Problem 18.48]{Kour}.

\begin{question}
Is it true that there are only finitely many integers which occur  as orders of  products of two class transpositions?
\end{question}

We denote by $CT_n$ a subgroup of $CT(\mathbb{Z})$ which is generated by horizontal class transpositions $\tau_{r_1(n),r_2(n)}$, $0 \leq r_1 \not= r_2 < n$. Obviously $\tau_{r_1(n),r_2(n)}$ permutes the residue classes modulo  $n$, and therefore $CT_n$ is isomorphic to the permutation  group $S_n$.

S. Kohl \cite{K} defined a monomorphism
$$
\varphi_m \colon S_m \to CT(\mathbb{Z})
$$
by the formula
$$
\sigma \mapsto (\sigma^{\varphi_m} \colon n \mapsto n+(n~mod(m))^{\sigma}) - (n~mod(m)),
$$
where we assume  that $S_m$ acts on the set $\{ 0, 1, \ldots, m-1 \}$.
%%%%%%%%%%%%%%%%%%%%%%%%%%%%%%%%%%

\bigskip

\section{Some properties of the group $CT(\mathbb{Z})$}
In \cite{K} was found the equality 
\begin{equation*} \label{Kol}
\tau_{0(2),1(2)} = \tau_{0(4),1(4)} \cdot \tau_{2(4),3(4)} \colon n \mapsto n + (-1)^n.
\end{equation*}

It is natural to find  the products 
$$
\tau_{0(6),1(6)} \cdot \tau_{2(6),3(6)} \cdot \tau_{4(6),5(6)},~~~\tau_{0(8),1(8)} \cdot \tau_{2(8),3(8)}. \cdot \tau_{4(8),5(8)}.  \cdot \tau_{6(8),7(8)}, \ldots 
$$

The following proposition answers this question.

\begin{proposition}\label{l}
Let $\tau_{r_1(m_1),r_2(m_2)}$ be a class transposition. Then

1) the permutation $\tau_{km_1+r_1(nm_1),km_2+r_2(nm_2)}$ is a class transposition for any natural $k$ and $n$ such that k $\leq n-1$;

2) for any natural $n$ the following equality holds
$$
\tau_{r_1(m_1),r_2(m_2)}=\prod\limits_{k=0}^{n-1}\tau_{km_1+r_1(nm_1),km_2+r_2(nm_2)}.
$$
\end{proposition}

\begin{proof}
1) Assume the opposite, i.e.  
$$
km_1+r_1(nm_1)\cap km_2+r_2(nm_2)\neq\emptyset.
$$
 Then there exist such integers $k_1$ and $k_2$ that the following equality holds
$$
km_1+r_1+nm_1k_1=km_2+r_2+nm_2k_2\Leftrightarrow r_1+m_1(k+nm_1k_1)=r_2+m_2(k+nk_2),
$$
but this means that $r_1(m_1)\cap r_2(m_2)\neq\emptyset$. This contradicts the fact that $\tau_{r_1(m_1),r_2(m_2)}$ is a  class transposition.

2) The proof follows from the  chain of equations: 
\begin{eqnarray*}
\tau_{r_1(m_1),r_2(m_2)} &=& \prod\limits_{c \in \mathbb{Z}} (r_1+m_1c,r_2+m_2c)\\
&=&\prod\limits_{s \in n\mathbb{Z}}\prod\limits_{k=0}^{n-1}(r_1+m_1(s+k),r_2+m_2(s+k))\\
&=&\prod\limits_{l \in \mathbb{Z}}\prod\limits_{k=0}^{n-1}(r_1+m_1(nl+k),r_2+m_2(nl+k))\\
&=&\prod\limits_{k=0}^{n-1}\prod\limits_{l \in \mathbb{Z}}(r_1+m_1k+m_1nl,r_2+m_2k+m_2nl))\\
&=&\prod\limits_{k=0}^{n-1}\tau_{km_1+r_1(nm_1),km_2+r_2(nm_2)}.
\end{eqnarray*}
By (1), every permutation $\tau_{km_1+r_1(nm_1),km_2+r_2(nm_2)}$ is a class transposition.
\end{proof}

We denote the subgroup that is generated by all of the horizontal class transpositions by $CT_{\infty}$.

\begin{proposition}\label{p2}
Let $\tau_{r_i(n_i),\tilde{r}_{i}(n_i)}$, where $i=1, 2, \ldots, m$ be  $m$  horizontal class transpositions. Then the order of their product can be found by the formula
$$
|\prod_{i=1}^m\tau_{r_i(n_i),\tilde{r}_{i}(n_i)}|=|\prod_{i=1}^m\prod_{k=0}^{p_i-1}(n_ik+r_i, n_ik+\tilde{r}_{i})|,
$$   
where $p_i=\frac{lcm(n_1, \ldots, n_m)}{n_i}$.
\end{proposition}

\begin{proof}
By Proposition \ref{l}, we have
$$
\prod_{i=1}^m\tau_{r_i(n_i),\tilde{r}_{i}(n_i)}=\prod_{i=1}^m\prod_{k=0}^{p_i-1}\tau_{n_ik+r_i(N),n_ik+\tilde{r}_{i}(N)},
$$
where $N = lcm(n_1, \ldots, n_m)$.

We assume that the permutation group $S_{N}$ acts on the set of  the residue classes modulo  $N$, i.e. on the set $\{ 0, 1, \ldots, N-1 \}$. Then the mapping 
$CT_{N} \to S_{N}$ which is given by the correspondence 
$$
\tau_{u_1(N),u_2(N)}\mapsto (u_1,u_2)
$$
is an isomorphism. Under the action of this isomorphism, the element $\prod_{i=1}^m\tau_{r_i(n_i),\tilde{r}_{i}(n_i)}$ goes to a permutation
$$
\prod_{i=1}^m\prod_{k=0}^{p_i-1}(n_ik+r_i, n_ik+\tilde{r}_{i}).
$$
\end{proof}

This proposition  reduces the study of the product of horizontal class transpositions to the study of the product of permutations of a finite set. 

\begin{example}
1) Consider the product of the class transpositions  $\tau_{0(2),1(2)}$ and $\tau_{0(3),1(3)}$. Using Proposition \ref{p2} we get 
\begin{eqnarray*}
|\tau_{0(2),1(2)}\cdot\tau_{0(3),1(3)}|&=&|(0, 1)(2, 3)(4, 5) \cdot (0, 1)(3, 4)|\\
&=&|(2, 3)(4, 5) \cdot (3, 4)|\\
&=&|(2, 4, 5, 3)|\\
&=&4,
\end{eqnarray*}
i.e. the order of the product $\tau_{0(2),1(2)}\cdot\tau_{0(3),1(3)}$ is equal to the order of the product of the permutation $(0, 1)(2, 3)(4, 5)$ on the permutation $(0, 1)(3, 4)$ in the group $S_6$.
Thus, class transpositions  $\tau_{0(2),1(2)}$ and $\tau_{0(3),1(3)}$ generate the dihedral group of order 8.
Direct computations give
$$
\tau_{0(2),1(2)}\cdot\tau_{0(3),1(3)} = \prod_{s \in \mathbb{Z}} (6 s) (1 + 6 s) (2 + 6 s, 4 + 6 s, 5 + 6 s, 3 + 6 s).
$$

2) Consider the product of the class transpositions $\tau_1 = \tau_{0(3),1(3)}$ and $\tau_2 = \tau_{2(4),3(4)}$. Using Proposition \ref{p2} we get
\begin{eqnarray*}
|\tau_{0(3),1(3)}\cdot\tau_{2(4),3(4)}|&=&|(0, 1)(3, 4)(6, 7)(9, 10) \cdot (2, 3)(6, 7)(10, 11)|\\
&=&|(0, 1)(3, 4)(9, 10) \cdot (2, 3)(10, 11)|\\
&=&|(0, 1)(2, 3, 4)(9, 11, 10)|\\
&=&6,
\end{eqnarray*}
i.e. the order of the product $\tau_{0(3),1(3)}\cdot\tau_{2(4),3(4)}$ is equal to the order of the product of the permutation $(0, 1)(3, 4)(6, 7)(9, 10)$ on the permutation $(2, 3)(6, 7)(10, 11)$ in the group $S_{12}$.
Thus the group $\langle \tau_{0(3),1(3)}, \tau_{2(4),3(4)}\rangle$ is isomorphic to the dihedral group of order 12. It can be shown that the following equality holds:
$$
\tau_1 \tau_2 = \prod_{s \in \mathbb{Z}} (12 s, 1+12s)(2 + 12s, 3+12s, 4+12s) (5+12s) (6+12s) (7+12s) (8+12s) (9+12s, 11+12s, 10+12s).
$$
\end{example}

Let us show that the group $CT_{\infty}$ coincides with the group $CT_{int}(\mathbb{Z})$ introduced in~\cite{K}. For this purpose, we recall some definitions from \cite{K}.
 The mapping $f \colon \mathbb{Z} \to \mathbb{Z}$ {\it is called affine on the set of residue classes} if there exists a positive integer $m$ such that the restriction of $f$ to all residue classes $r(m) \in \mathbb{Z} / m \mathbb{Z}$ is affine, i.e., defined as follows
$$
f|_{r(m)} \colon r(m) \to \mathbb{Z},~~n \mapsto (a_{r(m)} \cdot n + b_{r(m)}) / c_{r(m)},
$$
for some coefficients $a_{r(m)}$, $b_{r(m)}$,  $c_{r(m)} \in \mathbb{Z}$, depending on $r(m)$. The least possible $m$ is called the {\it modulus} of $f$ and is denoted by $Mod(f)$. We will assume that the greatest common divisor $gcd(a_{r(m)}, b_{r(m)}, c_{r(m)})$ is 1 and $c_{r(m)} > 0$. The {\it multiplier} of $f$ is called the least common multiple:
$$
lcm\{ a_{r(m)}~|~r(m) \in \mathbb{Z} / m \mathbb{Z} \},
$$
{\it the divisor} of $f$ is called
$$
lcm\{ c_{r(m)}~|~r(m) \in \mathbb{Z} / m \mathbb{Z} \}.
$$
A mapping $f$ is called {\it is integral} if its divisor is 1.  The  group which is  generated by integral class transpositions is denoted by $CT_{int}(\mathbb{Z})$.

Now we are ready to prove

\begin{proposition}\label{p3}
The group $CT_{\infty}$ coincides with the group $CT_{int}(\mathbb{Z})$.
\end{proposition}

\begin{proof}
Let $\tau_{r_1(m),r_2(m)}$ be a horizontal class transposition. Then 
\begin{equation*}
\tau_{r_1(m),r_2(m)}=
\begin{cases}
n+r_2-r_1,&\text{$n\in r_1(m)$},\\
n+r_1-r_2,&\text{$n\in r_2(m)$},
\end{cases}
\end{equation*}
i.e., $\tau_{r_1(m),r_2(m)}$ is an integral class transposition. 

Let $\tau_{r_1(m_1),r_2(m_2)}$ be an integral class transposition. Then 
\begin{equation*}
\tau_{r_1(m_1),r_2(m_1)}=
\begin{cases}
\frac{(n+r_2-r_1)m_2}{m_1},&\text{$n\in r_1(m_1)$},\\
\frac{(n+r_1-r_2)m_1}{m_2},&\text{$n\in r_2(m_2)$}.
\end{cases}
\end{equation*}
Since the class transposition $\tau_{r_1(m_1),r_2(m_2)}$ is an integral, by definition 
$$
lcm(c_{r_1(m_1)},c_{r_2(m_2)})=1.
$$
Hence $c_{r_1(m_1)}=c_{r_2(m_2)}=1$. It follows that $m_1=m_2$, i.e., the class transposition $\tau_{r_1(m_1),r_2(m_2)}$ is horizontal.

Hence, $CT_{\infty} = CT_{int}(\mathbb{Z})$.
\end{proof}

\bigskip

%%%%%%%%%%%%%%%%%%%%%%%%%%%%%%%%%%

\section{Graph of the product of two class transpositions}\label{Graph}

Consider a pair of class transpositions
$$
\tau_1 = \tau_{r_1(m_1),r_2(m_2)} = \prod\limits_{k \in \mathbb{Z}} (r_1+m_1k, r_2+m_2k),~~~~\tau_2 = \tau_{r_3(m_3),r_4(m_4)} = \prod\limits_{l \in \mathbb{Z}} (r_3+m_3l , r_4+m_4 l)
$$
and a pair of non-intersecting sets
$$
V_1 = \{ a_k, b_k ~| ~ k \in \mathbb{Z}\},~~~V_2 = \{ c_l, d_l ~| ~ l \in \mathbb{Z}\}.
$$
Let us define a graph $\Gamma(\tau_1,\tau_2)=(V,E)$ whose vertex set $V=V_1\sqcup V_2$ is an independent union of the sets $V_1$ and $V_2$. Let us define a mapping
 $$
 \mu \colon V \to Supp(\tau_1) \cup Supp(\tau_2) \subseteq \mathbb{Z},
 $$
such that 
$$
\mu(a_k) = r_1+m_1k,~~\mu(b_k) = r_2+m_2k,~~\mu(c_l) = r_3+m_3 l,~~\mu(d_l) = r_4+m_4 l.
$$
Obviously, the restriction of $\mu$ to each $V_i$, $i = 1, 2$, is a bijection.

The set of undirected edges $E$ consists of pairs $\{v_i, v_j \}\in E$ for which one of the following three cases holds

1) $ \mu(v_i) = \mu(v_j)$ and $v_i\in V_1, v_j\in V_2$ or vice versa $v_i\in V_2, v_j\in V_1$,

2) $\tau_1(\mu(v_i))= \mu(v_j)$ and $v_i,v_j\in V_1$,

3) $\tau_2(\mu(v_i))= \mu(v_j)$ and $v_i,v_j\in V_2$.

 The edges that satisfy condition 1) are called {\it edges of the first type}. The edges satisfying conditions 2) or 3) are called {\it edges of the second type}.
If $v$ is a vertex of the graph $\Gamma(\tau_1,\tau_2)$, then the symbol $d(v)$ will denote the number of vertices adjacent to $v$.

Also, to simplify the notation, we will write $\tau_i(u) = v$ instead of $\tau_i(\mu(u)) = \mu(v)$, $u, v \in V$, $i = 1, 2$.

\begin{lemma}\label{l2}
For any vertex $v$ of the graph $\Gamma(\tau_1,\tau_2)$ one of the following two statements holds:

1) vertex $v$ has valence 1 and, in this case, $v$ is incident to an edge of the second type;

2) vertex $v$ has valence 2 and, in this case, $v$ is incident to an edge of the first type and to an edge of the second type.
\end{lemma}

\begin{proof}
Let us assume that $v\in V_1$. The dual case is analyzed in the same way. Since $\mu(v) \in Supp(\tau_1)$, $\tau_1$ contains a transpose $(\mu(v), \mu(u))$ for some vertex $u \in V_1$. Hence, $v$ is incident to an edge of the second type. Obviously, such an edge is unique. If $\mu(v) \not \in Supp(\tau_2)$, then $v$ has valence 1 and we obtain statement 1).  If $\mu(v) \in Supp(\tau_2)$, then there is a unique vertex $w \in V_2$ such that $\mu(w) = \mu(v)$. In this case $v$ has valence 2 and we obtain statement 2). Since there are no other possibilities, the lemma is proved.
\end{proof}

Using this lemma it is easy to classify the connected components of the graph $\Gamma(\tau_1,\tau_2)$. If a connected component is finite, i.e., it contains a finite number of vertices and all vertices have valence 2, then we choose a vertex and denote it by $v_1$. Then this connected component is given by the sequence 
$$
v_1 v_2\dots v_{n-1}v_1,
$$
in which a pair of neighbouring vertices are connected by an edge. Then, by specifying the connected components by a sequence of vertices, we consider that neighbouring vertices are connected by an edge. Connected components of this kind will be called connected components of the  {\it first type}. If there are vertices of valence 1 in a finite connected component, it is easy to see that there are exactly two of them. We denote one of them by $v_1$ and the other by $v_n$ such that the inequality $\mu(v_1)<\mu(v_n)$ holds. Then this connected component can be written as 
$$
v_1v_2\dots v_n.
$$
Connected components of this kind will be called connected components  of the {\it  second type}. If all vertices of a connected component have valence 2 and their number is infinite, then we denote by $v_0$ any vertex of this component. We will write this connected component in the form 
$$
\prod_{i\in\mathbb{Z}}v_i.
$$
We will call such connected components  of the {\it  third type}. If a connected component contains infinitely many vertices and a vertex of valence 1 is found, then it is the only one. We denote it by $v_1$ and write this connected component as
$$
\prod_{i=1}^{\infty}v_i.
$$
Connected components of this kind will be called connected components of the {\it fourth type}.

The following theorem establishes the relationship between the connected components of a graph and the orbits of the product of a pair of class transpositions.

\begin{theorem}\label{t1}
There exists an injection $\psi$ of the set of connected components $\{S_i\}_{i\in I}$ of the graph $\Gamma(\tau_1,\tau_2)$ into permutations belonging to the product $\tau_1\cdot \tau_2$ such that $Supp(\psi(S_i))\cap Supp(\psi(S_j))=\emptyset$ if $i\neq j$. In this case. 
$$
\prod_{i\in I}\psi(S_i)=\tau_1\cdot\tau_2.
$$
\end{theorem}

\begin{proof}
Let $S$ be a connected component of the graph $\Gamma(\tau_1,\tau_2)$. Depending on the type of $S$, there are several cases to consider.

{\it Case 1:} $S$  is a component of the first type.  Then
$$
S=v_1v_2\dots v_{n-1}v_1.
$$
We can assume that $v_1\in V_1$ and $v_1$ is incident to an edge of the second type. Then 
$$
v_1\xrightarrow{\tau_1}v_2=v_3\xrightarrow{\tau_2}v_4=\ldots = v_{n-1}\xrightarrow{\tau_2}v_n=v_1.
$$
In this case, $n=4l$ and the permutation has the form
$$
(v_1,v_4,v_8,\ldots,v_{4k},\ldots ,v_{n-8},v_{n-4})(v_{n-2},v_{n-6},v_{n-10}\ldots v_{n-2-4s}\dots ,v_6,v_2),
$$
where $1\leq k\leq l-1$, $0 \leq s \leq l-1$.

Thus, each connected component of the first type of  $\Gamma(\tau_1,\tau_2)$ corresponds to the product of two independent cycles in the product $\tau_1\cdot\tau_2$.

{\it Case 2:} $S$ is a component of the second type. In this case, 
$$
S=v_1v_2\dots v_n,
$$
where $v_1$ and $v_n$ are 1-valent vertices. Then each of them is incident to an edge of the first type. We assume that $v_1\in V_1$. The dual case is dealt with similarly. Then, depending on the parity of $n$, either
$$
v_1\xrightarrow{\tau_1}v_2=v_3\xrightarrow{\tau_2}v_4=\ldots =v_{n-1}\xrightarrow{\tau_2}v_n,
$$
or
$$
v_1\xrightarrow{\tau_1}v_2=v_3\xrightarrow{\tau_2}v_4=\ldots =v_{n-1}\xrightarrow{\tau_1}v_n.
$$

It is easy to see that in the first case $n=4l$ and the permutation corresponding to the connected component $S$ has the form 
$$
(v_{n-1},\ldots,v_{4k-1},\ldots v_7,v_3,v_1,v_4,v_8,\ldots,v_{4s},\ldots,v_n),
$$
and in the second case $n=4l+2$  the permutation has the form 
$$
(v_{n-3},\ldots ,v_{4k-1},\ldots ,v_4,v_3,v_1,v_4,v_8,\ldots,v_{4s},\ldots,v_{n-2},v_n),
$$
where $1\leq k\leq l$, $1\leq s\leq l$ in both cases.

Thus, each connected component of the second type of  $\Gamma(\tau_1,\tau_2)$ corresponds to a cycle in the product $\tau_1\cdot\tau_2$.

{\it Case 3:} $S$ is a component of the third type. Then 
$$
S=\prod_{i\in \mathbb{Z}}v_i.
$$
Since any vertex can be chosen as the vertex $v_0$, we choose such a vertex that 
$$
\ldots=v_{-3}\xrightarrow{\tau_1}v_{-2}=v_{-1}\xrightarrow{\tau_2}v_0=v_1\xrightarrow{\tau_1}v_2=v_3\xrightarrow{\tau_2}v_4=\ldots
$$
In this case, the permutation corresponding to the component $S$  has the form 
$$
(\ldots,v_{3+4t},\ldots,v_{11},v_7,v_2,v_{-2},v_{-6},\ldots,v_{2-4l},\ldots)(\ldots,v_{1-4k},\ldots,v_{-7},v_{-3},v_1,v_4,v_8,\ldots,v_{4s},\ldots),
$$
where $1\leq t$, $1\leq l$, $1\leq k$, $1\leq s$.

Thus, each connected component of the third type of  $\Gamma(\tau_1,\tau_2)$ corresponds to the product of two independent cycles of infinite length in the product $\tau_1\cdot\tau_2$.

{\it Case 4:} $S$ is  a component of the fourth type. Then
$$
S=\prod_{i=1}^{\infty}v_i.
$$
By Lemma \ref{l2}, vertex $v_1$ is incident to an edge of the second type. Let us assume that $v_1\in V_1$. Then 
$$
v_1\xrightarrow{\tau_1}v_2=v_3\xrightarrow{\tau_2}v_4=\ldots 
$$
In this case, the permutation will be
$$
(\ldots,v_{4k-1},\ldots v_7,v_3,v_1,v_4,v_8,\ldots,v_{4s},\ldots),
$$
where $1\leq k$, $1\leq s$.

Thus, each connected component of the fourth type of  $\Gamma(\tau_1,\tau_2)$ corresponds to an infinite cycle in the product $\tau_1\cdot\tau_2$.

If $Supp(\psi(S_i))\cap Supp(\psi(S_j))\neq\emptyset$ for  $i\neq j$, then the connected components $S_i$ and $S_j$ have common vertices lying in the intersection of $Supp(\psi(S_i))\cap Supp(\psi(S_j))$. So $S_i=S_j$.
\end{proof}

Given this theorem, it is sufficient to describe all the lengths of the connected components and their types to describe the orders of the products of two class transpositions. 

\bigskip

%%%%%%%%%%%%%%%%%%%%%%%%%%%%%%%%%%

\section{The product of two horizontal class transpositions}

\medskip

If $\tau_{r_1(n),r_2(n)}$ is a horizontal class transposition, then we assume that $r_1 < r_2$. In view of the inequality $r_1 < r_2 < n$, we conclude that for any integer $k$ the inequalities are true 
$$
r_1+nk<r_2+nk,~~r_2+n k<r_1+n(k+1).
$$ 

Let us choose a pair of horizontal class transpositions:
$$
\sigma=\tau_{r_1(n),r_2(n)},~~\eta=\tau_{r_3(m),r_4(m)},~~ r_1<r_2,~~r_3<r_4.
$$
Recall (see \S \ref{Graph}) that the set of vertices $V$ of the graph $\Gamma(\sigma,\eta)$ consists of the subsets
$$
V_1 = \{ a_k, b_k \,|\, k \in \mathbb{Z}\},~~~V_2 = \{ c_k, d_k \,|\,k \in \mathbb{Z}\}.
$$ 
The function $\mu$ acts on the vertices as follows:
$$
\mu(a_k)=r_1+nk, ~~\mu(b_k)=r_2+nk, ~~\mu(c_k)=r_3+mk, ~~\mu(d_k)=r_4+mk.
$$
Hence, $\sigma$ and $\eta$ are of the form:
$$
\sigma=\prod_{k\in \mathbb{Z}}(\mu(a_k),\mu(b_k)),~~~~~
\eta=\prod_{k\in \mathbb{Z}}(\mu(c_k),\mu(d_k)).
$$

We can give a geometric interpretation of the graph $\Gamma(\tau_1,\tau_2)$. Consider the Cartesian plane $\mathbb{R}^2$ and place the vertices of  $V_1$ on the line $y = 1$:
$$
a_k = (r_1+m_1k, 1),~~b_k = (r_2+m_2k, 1),~~ k \in \mathbb{Z}.
$$
Similarly, let us place the vertices of $V_2$ on the line $y = 0$:
$$
c_k = (r_3+m_4 l, 0),~~d_k = (r_4+m_4 l, 1),~~ l \in \mathbb{Z}.
$$
The edges will be represented by segments connecting the corresponding vertices. Two subgraphs of this graph will be called {\it symmetric} if one of them can be obtained from the other by reflection with respect to a horizontal or vertical line and  renaming the vertices. In the following, to simplify the notation, we will write $a_k, b_k, c_k, d_k$ instead of $\mu(a_k), \mu(b_k), \mu(c_k), \mu(d_k)$.

\begin{lemma}\label{l3}
The connected components of the graph $\Gamma(\sigma,\eta)$ do not contain a subgraph of the form
$$
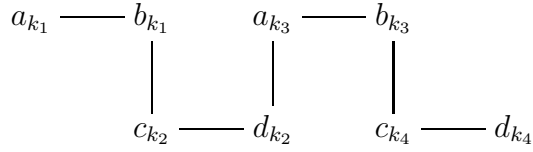
\xymatrix{{a_{k_1}} \ar@{-}[r] & {b_{k_1}} \ar@{-}[d]&{a_{k_3}}\ar@{-}[r]\ar@{-}[d] & {b_{k_3}}\ar@{-}[d]\\
&{c_{k_2}}\ar@{-}[r] & {d_{k_2}} & {c_{k_4}}\ar@{-}[r] & {d_{k_4}}
}~~~~~~
$$
or symmetric to it.
\end{lemma}

\begin{proof}
It is enough to give a proof for the graph shown in the formulation of the lemma, the symmetric cases are dealt with similarly. Suppose the contrary, i.e., that we find a connected component  $S$ which has a subgraph of the form
$$
a_{k_1}b_{k_1}c_{k_2}d_{k_2}a_{k_3}b_{k_3}c_{k_4}d_{k_4}.
$$
In this case
$$
b_{k_1}=c_{k_2}\xrightarrow{\eta}d_{k_2}=a_{k_3}\xrightarrow{\sigma}b_{k_3}=c_{k_4}
$$
and the system of Diophantine equations

\begin{equation*}
 \begin{cases}
   r_2+nx_1=r_3+mx_2, 
   \\
   r_4+mx_2=r_1+nx_3,
   \\
   r_2+nx_3=r_3+mx_4
 \end{cases}
\end{equation*} 
is solvable and has a solution
$$
x_1=k_1,~~x_2=k_2,~~x_3=k_3,~~x_4=k_4.
$$

We see that $nk_1+r_2-r_3$ is divisible by $m$, and $r_4-r_1+r_2-r_3$ is divisible by $n$ and by $m$. Add to the system an equation of the form
$$
r_4+mx_4=r_1+nx_5.
$$
Then $x_4=k_4, x_5=k_5=k_1+2\frac{r_4-r_1+r_2-r_3}{n}$ gives a solution of systems consisting of four equations. Add to this system the equation  
$$
r_2+nx_5=r_3+mx_6.
$$
Assuming $x_5=k_5, x_6=k_6=\frac{nk_1+r_2-r_3}{m}+2\frac{r_4-r_1+r_2-r_3}{m}$, we obtain a solution to the system of five equations. Continuing this process, we construct a solution to the following infinite system of Diophantine equations: 
\begin{equation*}
 \begin{cases}
   r_2+nx_1=r_3+mx_2, 
   \\
   r_4+mx_2=r_1+nx_3,
   \\
   r_2+nx_3=r_3+mx_4,
   \\
   r_4+mx_4=r_1+nx_5,
   \\
   r_2+nx_5=r_3+mx_6,
   \\
   r_4+mx_6=r_1+nx_7,
   \\
   \cdots\cdots\cdots\cdots\cdots\cdots\cdots
 \end{cases}
\end{equation*}
Hence we conclude that the connected component $S$ contains an infinite sequence 
$$
b_{k_1}c_{k_2}d_{k_2}a_{k_3}b_{k_3}c_{k_4}d_{k_4}a_{k_5}\dots,
$$
and hence the product $\sigma\cdot\eta$ contains a cycle of infinite length and hence has infinite order, but this contradicts the local finiteness of the group $CT_{\infty}$.
\end{proof}

\begin{lemma}\label{l4}
Every connected component of the first type of the  graph $\Gamma(\sigma,\eta)$, has the form
$$
\xymatrix{
{a_{k_1}} \ar@{-}[r] \ar@{-}[d] & {b_{k_1}}\ar@{-}[d]\\
{c_{k_2}} \ar@{-}[r] & {d_{k_2}}
}
$$
\end{lemma}

\begin{proof}
The statement follows from the geometric interpretation of $\Gamma(\sigma,\eta)$ and from the fact that $b_k<a_{k+1}$, $c_s<d_{s+1}$ for any integers $k$ and $s$.
\end{proof}

\begin{lemma}\label{l5}
The graph $\Gamma(\sigma,\eta)$ does not contain  connected components of the following form  
$$
\xymatrix{
{a_{k_1}}\ar@{-}[r]\ar@{-}[d]& {b_{k_1}} & {a_{k_3}}\ar@{-}[r]&{b_{k_3}}\\
{c_{k_2}}\ar@{-}[rr]& &{d_{k_2}}\ar@{-}[u]&{c_{k_4}}\ar@{-}[r]\ar@{-}[u]&{d_{k_4}}
}~~~~~~
\xymatrix{
{a_{k_1}}\ar@{-}[r]\ar@{-}[d]& {b_{k_1}} & {a_{k_3}}\ar@{-}[rr]& & {b_{k_3}}\\
{c_{k_2}}\ar@{-}[rr]& &{d_{k_2}}\ar@{-}[u]&{c_{k_4}}\ar@{-}[r]&{d_{k_4}}\ar@{-}[u]
}
$$
or symmetric to them.
\end{lemma}

\begin{proof}
It suffices to give a proof for the graphs depicted in the formulation of the lemma; symmetric graphs are dealt with similarly. 

Suppose that  $\Gamma(\sigma,\eta)$ contains a subgraph of the following form:  
$$
\xymatrix{
{a_{k_1}}\ar@{-}[r]\ar@{-}[d]& {b_{k_1}} & {a_{k_3}}\ar@{-}[r]&{b_{k_3}}\\
{c_{k_2}}\ar@{-}[rr]& &{d_{k_2}}\ar@{-}[u]&{c_{k_4}}\ar@{-}[r]\ar@{-}[u]&{d_{k_4}}
}
$$
Then the following system of Diophantine equations is solvable:
\begin{equation*}
 \begin{cases}
   r_1+nx_1=r_3+mx_2, 
   \\
   r_4+mx_2=r_1+nx_3,
   \\
   r_2+nx_3=r_3+mx_4.
 \end{cases}
\end{equation*}
The solution is
$$
x_1=k_1,~~x_2=k_2,~~x_3=k_3,~~x_4=k_4.
$$
Substitute the solution into the system and express $nk_3$ through $k_1$. We obtain the equality
$$
r_4-r_3+nk_1=nk_3,
$$
so $r_4-r_3$ is divisible by $n$. Let us express $k_4$ through $k_1$. We obtain that
$$
r_2-2r_3+r_4+nk_1=mk_4.
$$

Let us add the  equation:
$$
r_4+mx_4=r_2+nx_5
$$
 to the system.
By virtue of the previous equality,
$$
2r_4-2r_3+nk_1=nx_5.
$$
Since $r_4-r_3$ is divisible by $n$, then $x_5=\frac{2r_4-2r_3}{n}+k_1$ is a solution to the system of four equations. So the original subgraph is always contained in a connected component:  
$$
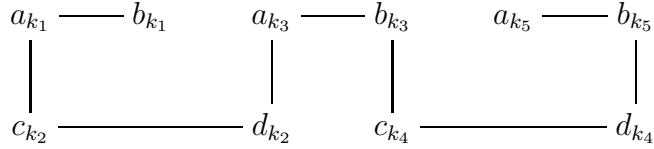
\xymatrix{
{a_{k_1}}\ar@{-}[r]\ar@{-}[d]& {b_{k_1}} & {a_{k_3}}\ar@{-}[r]&{b_{k_3}}&a_{k_5}\ar@{-}[r]&b_{k_5}\\
{c_{k_2}}\ar@{-}[rr]& &{d_{k_2}}\ar@{-}[u]&{c_{k_4}}\ar@{-}[rr]\ar@{-}[u]&&{d_{k_4}\ar@{-}[u]}
}
$$
Suppose that the graph $\Gamma(\sigma,\eta)$ contains a connected component of the form:
$$
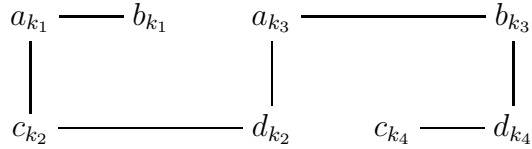
\xymatrix{
{a_{k_1}}\ar@{-}[r]\ar@{-}[d]& {b_{k_1}} & {a_{k_3}}\ar@{-}[rr]& & {b_{k_3}}\\
{c_{k_2}}\ar@{-}[rr]& &{d_{k_2}}\ar@{-}[u]&{c_{k_4}}\ar@{-}[r]&{d_{k_4}}\ar@{-}[u]
}
$$
Then the following system of Diophantine equations is solvable:
\begin{equation*}
 \begin{cases}
   r_1+nx_1=r_3+mx_2, 
   \\
   r_4+mx_2=r_1+nx_3,
   \\
   r_2+nx_3=r_4+mx_4.
 \end{cases}
\end{equation*}
From the existence of a solution to the first two equations, we obtain that $r_4-r_3$ is divisible by $n$, and from the existence of a solution to the second and third equations, we obtain that $r_2-r_1$ is divisible by $m$. But then the following inequalities hold
$$
m>|r_4-r_3|\geq n,\,\,n>|r_2-r_1|\geq m.
$$
We have reached a contradiction.
\end{proof}

\begin{theorem}\label{t2}
Every connected component of the graph $\Gamma(\sigma,\eta)$ is a graph of one of the following types: 

$$
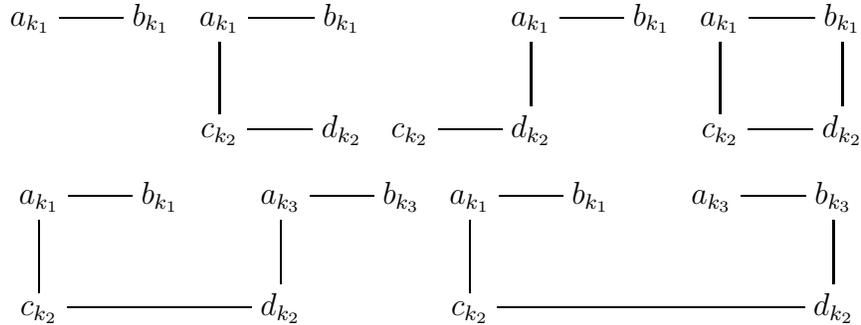
\xymatrix{ {a_{k_1}} \ar@{-}[r] & {b_{k_1}}}~~~~~~
\xymatrix{
{a_{k_1}} \ar@{-}[r]\ar@{-}[d] & {b_{k_1}}\\
{c_{k_2}} \ar@{-}[r] & {d_{k_2}}
}~~~~~~
\xymatrix{
&{a_{k_1}} \ar@{-}[r]\ar@{-}[d] & {b_{k_1}}\\
{c_{k_2}} \ar@{-}[r] & {d_{k_2}}
}~~~~~~
\xymatrix{
{a_{k_1}} \ar@{-}[r] \ar@{-}[d] & {b_{k_1}}\ar@{-}[d]\\
{c_{k_2}} \ar@{-}[r] & {d_{k_2}}
}~~~~~~
$$
$$
\xymatrix{
{a_{k_1}}\ar@{-}[r]\ar@{-}[d]& {b_{k_1}} & {a_{k_3}}\ar@{-}[r]& {b_{k_3}}\\
{c_{k_2}}\ar@{-}[rr]&&{d_{k_2}}\ar@{-}[u]
}~~~~~~
\xymatrix{
{a_{k_1}}\ar@{-}[r]\ar@{-}[d]& {b_{k_1}} & {a_{k_3}}\ar@{-}[r]& {b_{k_3}}\\
{c_{k_2}}\ar@{-}[rrr]&&&{d_{k_2}}\ar@{-}[u]
}~~~~~~
$$
$$
\xymatrix{
{a_{k_1}}\ar@{-}[r]\ar@{-}[d]& {b_{k_1}} & {a_{k_3}}\ar@{-}[r]& {b_{k_3}}&{a_{k_5}}\ar@{-}[r]&{b_{k_5}}\\
{c_{k_2}}\ar@{-}[rr]&&{d_{k_2}}\ar@{-}[u]&{c_{k_4}}\ar@{-}[rr]\ar@{-}[u]&&{d_{k_4}\ar@{-}[u]}
}
$$
or symmetric to them.
\end{theorem}

\begin{proof}
It follows from Lemma \ref{l3}, Lemma \ref{l4}, Lemma \ref{l5} and the geometric interpretation of the graph $\Gamma(\sigma,\eta)$.
\end{proof}

This theorem completely describes the connected components of the graph $\Gamma(\sigma,\eta)$. As a consequence, we obtain the main result of this paper.

\begin{corollary}\label{c1}
An order of the  product rder of a pair of horizontal class transpositions 
belongs to the following set $\{1,2,3,4,6,12\}$. Moreover, for any number from this set there is a pair of horizontal class transpositions such that the  order of their  product is equal to this number.
\end{corollary}

\begin{proof}
The first part follows from Theorem \ref{t1} and Theorem \ref{t2}, and the second part follows from the following equalities derived from Proposition \ref{p2}

$$
|\tau_{0(2),1(2)}\cdot \tau_{0(4),2(4)}|=2,
$$
$$
|\tau_{0(3),1(3)}\cdot \tau_{0(3),2(3)}|=3,
$$
$$
|\tau_{0(2),1(2)}\cdot \tau_{0(3),1(3)}|=4,
$$
$$
|\tau_{0(2),1(2)}\cdot \tau_{0(3),2(3)}|=6,
$$
$$
|\tau_{0(3),1(3)}\cdot \tau_{0(4),2(4)}|=12.
$$
\end{proof}

\bigskip

%%%%%%%%%%%%%%%%%%%%%%%%%%%%%%%%%%

\section{Open questions}

\medskip

As we know, the group $CT_{\infty}$ which is generated by horizontal class transpositions,  also can be generated by subgroups $CT_k$ which are   isomorphic to the permutation
 groups $S_k$, $k = 2, 3, \ldots$. It is interesting to find out: ehat can we say on a  group $CT_{(k)}$ which is generated by the subgroups $CT_2, CT_3, \ldots, CT_k$?

The orders of the following subgroups were found in the computer algebra system GAP: 
$$
|\langle CT_2, CT_3\rangle|=5!
$$
$$
|\langle CT_2, CT_3, CT_4\rangle|=12!
$$
$$
|\langle CT_3, CT_4\rangle|=12!
$$
$$
|\langle CT_2, CT_5\rangle|=10!
$$
$$
|\langle CT_3, CT_5\rangle|=15!
$$
$$
|\langle CT_2, CT_3, CT_5\rangle|=30!
$$
The isomorphism $\langle CT_2, CT_3, CT_4\rangle \leq CT_{12}\cong S_{12}$ was also established using GAP. Similar isomorphisms can be verified for other examples besides the first one. These observations allow us to formulate the following conjecture.

\begin{conjecture}
For  $k > 3$, the group $CT_{(k)}$ is isomorphic to the permutation group $S_N$, where $N$ is the least common multiple of the numbers $2, 3, \ldots, k$.
\end{conjecture}

Using GAP or direct computation, it is easy to check the isomorphism 
$$
CT_{(3)} = \langle CT_2, CT_3\rangle\cong S_5.
$$
We can give the following explanation of the fundamental difference between the case $k=3$ and the case $k>3$.  Obviously, if we consider a subgroup in the group $S_n$ consisting of all substitutions leaving some symbol in place, then this subgroup will be isomorphic to $S_{n-1}$. Such an embedding of $S_{n-1}$ into $S_n$ is called a standard embedding. When $n=6$, there exists also a non-standard embedding $S_5$ into $S_6$, as demonstrated by the group $CT_{(3)}$. Indeed, the group $CT_{(3)}$ is embedded into $S_6$ and this embedding is given by a mapping defined on the generators:
$$
\tau_{0(2),1(2)}\mapsto a = (0,1) (2, 3) (4, 5),~~ \tau_{0(3),1(3)} \mapsto b = (0, 1) (3, 4),~~ \tau_{0(3),2(3)} \mapsto c = (0, 2)(3, 5),
$$
$$
\tau_{1(3),2(3)} \mapsto d = (1, 2)(4, 5).
$$
Obviously, the group $\langle a, b, c, d \rangle$ is a subgroup of  $S_6$ and it  does not fix any symbol from the set $\{ 0, 1, 2, 3, 4, 5 \}$, and therefore gives a non-standard embedding of  $S_5$ into $S_6$.

The following questions also seem interesting.

\begin{question}
Is the group $CT_{\infty}$ have finite width with respect to the set of horizontal class transpositions? In other words, is there a natural number $n$ such that every permutation from $CT_{\infty}$ can be represented as a product of at most $n$ horizontal class transpositions?
\end{question}

\begin{question}
What can we say on  the subgroups which are  generated by three class transpositions?  The same question is asked when two from these class transpositions are commute. 
\end{question}

\bigskip

%%%%%%%%%%%%%%%%%%%%%%%%%%%%%%%%%%
\section{Acknowledgement}
The work of V.G. Bardakov and A. L. Iskra was performed at the Saint Petersburg Leonhard Euler
International Mathematical Institute and supported by the Ministry of Science and Higher
Education of the Russian Federation (agreement no. 075–15–2022–287).

\bigskip

%%%%%%%%%%%%%%%%%%%%%%%%%%%%%%%%%%

\end{document}